\documentclass[12pt]{amsart}
\usepackage{amsmath,amsthm,latexsym,amscd,amsbsy,amssymb,url}
\usepackage[all]{xypic}
 \relax

\chardef\bslash=`\\ 

\makeatletter
\def\verbatim{\interlinepenalty\@M \@verbatim
  \leftskip\@totalleftmargin\advance\leftskip2pc
  \frenchspacing\@vobeyspaces \@xverbatim}
\makeatother
\newtheorem{thm}{Theorem}[section]
\newtheorem{cor}[thm]{Corollary}
\newtheorem{lem}[thm]{Lemma}
\newtheorem{pro}[thm]{Proposition}
\newtheorem{que}[thm]{Question}

\theoremstyle{definition}

\theoremstyle{remark}

\numberwithin{equation}{section}

\begin{document}


\title
{Fixed-point free maps of Euclidean spaces}
\author{R.~Z.~Buzyakova}
\address{Department of Mathematics and Statistics,
The University of North Carolina at Greensboro,
Greensboro, NC, 27402, USA}
\email{rzbouzia@uncg.edu}
\author{A.~Chigogidze}
\address{Department of Mathematics and Statistics,
The University of North Carolina at Greensboro,
Greensboro, NC, 27402, USA}
\email{chigogidze@uncg.edu}


\begin{abstract}{Our main result states that every fixed-point free continuous self-map of ${\mathbb R}^{n}$ is colorable.
This result can be re-formulated as follows:
A  continuous map $f: {\mathbb R}^{n}\to {\mathbb R}^{n}$ is fixed-point free iff
$\widetilde f: \beta {\mathbb R}^{n}\to \beta {\mathbb R}^{n}$ is fixed-point free.
We also obtain a generalization of this fact and present some examples.}
\end{abstract}

\maketitle
\markboth{R.Z.Buzyakova, A. Chigogidze}{Fixed-point free maps}
{ }

\section{Introduction}\label{S:intro}

It is known that for a continuous map $f \colon X \to X$ the set of all fixed points 
$\operatorname{Fix}(\widetilde{f})$ of its Stone-\v{C}ech extension 
$\widetilde{f} \colon \beta X \to \beta X$ may differ from 
$\operatorname{cl}_{\beta X}\operatorname{Fix}(f)$. In particular, $\operatorname{Fix}(\widetilde{f})$ 
could be non-empty for a fixed-point free $f$,
which is equivalent to  non-colorability of $f$ if $X$ is normal. Since the just used notion of colorability may not be
widely known let us define it.
\par\bigskip\noindent
{\bf Definition of Colorability.} {\it
$f: X\subset Y\to Y$ is {\it colorable} if one can 
cover $X$ by finitely many closed sets each of which misses its image under $f$.
}
\par\bigskip\noindent
Metric examples (constructed by Krzysztof Mazur and van Douwen) of fixed-point 
free non-colorable maps can be found 
in \cite{K}, \cite{D}, \cite{KS}, \cite{VM2}. 
It is interesting to note that these examples are either infinite-dimensional or non-locally compact. 
On the other hand, it is known that: (a) every fixed-point free autohomeomorphism $f$ of a 
finite-dimensional paracompact space $X$ has the fixed-point free extension $\widetilde{f}$ 
(see \cite{D}), and (b) every continuous fixed-point free map $f$ of a zero-dimensional sigma-compact 
space $X$ has the  fixed-point free extension $\widetilde{f}$ (see \cite{KS}). We also mention here 
recently obtained result \cite{R} stating colorability of {\it every} continuous fixed-point free self-map of the real line ${\mathbb R}$.  

Below we show (Theorem \ref{T:EUCLIDEAN}) that every continuous fixed-point free self-map of the Euclidean space 
${\mathbb R}^{n}$ is colorable. In fact, we show more by considering not only 
self-maps $f \colon X \to X$, but also maps $f \colon X \to Y$, defined on closed subspaces $X$ of $Y$. 
This version shows itself helpful in various situations and in our opinion is of  interest. The most 
general statement we present in this paper is Theorem \ref{THEOREM:GENERAL} which states that every 
continuous fixed-point free map $f \colon X \to Y$, defined on a closed subspace $X$ of an at most $n$-dimensional 
locally compact and paracompact space $Y$, is colorable in at most $n+3$ colors. 
We also note that a continuous map $f \colon X \to Y$, defined on a closed subspace $X$ of an at most $n$-dimensional 
locally compact and paracompact space $Y$, is fixed-point free iff $\widetilde f: \beta X\to \beta Y$ is fixed-point free. 
In Section \ref{S:examples}, 
we construct a non-colorable fixed-point free autohomeomorphism of the separable Hilbert space $\ell_{2}$ and 
a non-colorable continuous fixed-point free self-map of the universal $n$-dimensional N\"{o}beling space $N_{n}^{2n+1}$. 
The case $n=0$ produces (and uses) the same example as that of K. Mazur's since $N_{0}^{1}$ is the space of 
irrationals. We conclude by proving (Proposition \ref{P:iff}) that similar examples exist on every zero-dimensional 
non-sigma-compact Polish space. 

A good account of the results related to this subject as well as historical comments are given in \cite[Section 3.2]{VM2}. 
Our notation and terminology related to absolute extensors in dimension $n$, $n$-soft maps, universal N\"{o}beling 
spaces and inverse spectra follow \cite{chibook}.

When discussing the $n$-th power ${\mathbb R}^n$ of ${\mathbb R}$ we agree
that ${\mathbb R}^0$ is the singleton $\{0\}$. 
We finish the introduction by defining colorability-related concepts used in the paper. Once again,
a continuous map $f:X\to Y$ is {\it colorable} if there exists a finite 
cover $\mathcal F$ of $X$ by its closed subsets such that $f(F)$ misses $F$ for every $F\in {\mathcal F}$.
If $\mathcal F$ has $n$ elements we say that $f$ is {\it colorable in $n$ colors}. Every closed
set $F\subset X$ with the property that $f(F)\cap F=\emptyset $ is a {\it color} (for $f$).
All maps under discussion are continuous.

\section{Fixed-point free maps of Euclidean spaces}\label{S:es}

\par\bigskip
The following lemma consists of common facts about colorable maps that are proved implicitly for self-maps in other
papers on the subject. Since the proof is a  straightforward verification we will only sketch it.

\par\bigskip\noindent
\begin{lem}\label{L:COMMONFACTS}
Let $X$ be a closed subspace of $Z$ and $f:X\to Z$ a continuous map.
\begin{enumerate}
	\item If $Z$ is normal, $A$ is  a color, and $A$ misses $\overline {f(A)}$; then there is an open neighborhood $U$
of $A$ whose closure is a color.
	\item If $A\subset Z$ and $f(A)\cap A=\emptyset$ then $f^{-1}(A)\cap A=\emptyset$.

	\item If $\mathcal G$ is a closed cover of $Z$ such that $f(G)\cap G=\emptyset$
then 
${\mathcal F}=\{f^{-1}(G): G\in {\mathcal G}\}$ is a cover of $X$ by colors and $\overline {f(F)}\cap F=\emptyset$
for all $F\in {\mathcal F}$.
\end{enumerate}
\end{lem}
\begin{proof}
To prove 2, assume the contrary and pick $x$ in $f^{-1}(A)\cap A=\emptyset$. Since $f^{-1}(A)$ is a subset
of $X$, the value $f(x)$ is defined. Hence $f\circ f^{-1}(A)$ meets $f(A)$, contradicting $f(A)\cap A=\emptyset$.

To prove 3, apply 2 to conclude that $f^{-1}(G)$ is a color. Since $\mathcal G$ is a cover of $Z$ and $f$ is defined on $X$ only,
${\mathcal F}$ is a cover of $X$. Finally, let us show that $\overline {f(F)}\cap F=\emptyset$.
We have $F = f^{-1}(G)$ for $G\in {\mathcal G}$. Then $\overline {f(F)}\cap F$ equals 
$\overline {f\circ f^{-1}(G)}\cap f^{-1}(G)$, which is a subset of $\overline G\cap f^{-1}(G)$. Since $G$ is closed
$\overline G\cap f^{-1}(G)$ equals $G\cap f^{-1}(G)$, which is empty by 2.
\end{proof}

\par\bigskip
We will use the above facts quite often without formally referring to the statement.

\par\bigskip\noindent

\begin{lem}\label{L:ENLARGE}
Let $X$ be a normal space of dimension at most $n$. Suppose $\{A_i,B_i:i\leq n+1\}$ is a 
family of functionally closed sets such that $A_i\cap B_i=\emptyset$.
Then there exists a functionally closed cover $\{\tilde A_i,\tilde B_i:i\leq n+1\}$ of $X$
with the following properties:
\begin{enumerate}
	\item $\tilde A_i\cap \tilde B_i =\emptyset$; and
	\item $\tilde A_i\cap Z= A_i$ and $\tilde B_i\cap Z=B_i$, where $Z=\bigcup\{A_i,B_i:i\leq n+1\}$. 
\end{enumerate}
\end{lem}
\begin{proof}
For each $i = 1,\dots, n+1$, consider a map $f_{i} \colon X \to I=[0,1]$ such that 
$f^{-1}_{i}(0) = A_i$ and $f^{-1}_{i}(1) = B_i$. The diagonal product 
$f = \triangle f_{i} \colon X \to [0,1]^{n+1}$ sends the union $\cup(A_{i}\cup B_{i})$ 
into $\partial I^{n+1}$. Since $\dim X \leq n$ and since $\partial I^{n+1} \simeq S^{n}$ 
is an absolute extensor in dimension $n$ for normal spaces, it follows that the map 
$f|\cup(A_{i}\cup B_{i}) \colon \cup(A_{i}\cup B_{i}) \to \partial I^{n+1}$ has an 
extension $g \colon X \to \partial I^{n+1}$. The sets $\widetilde{A}_{i}$, $\widetilde{B}_{i}$ 
can now be defined as inverse images under the map $g$ of the opposite faces of $\partial I^{n+1}$. 
The needed properties hold by construction.
\end{proof}

\par\bigskip
To prove our main result for ${\mathbb R}^{n}$ and a generalization to n-dimensional locally compact paracompact spaces
we encounter several cases that can be handled by one approach realized in 
Lemma \ref{L:GOOD}. To formulate the lemma to a required generality we need the following definition.


\par\bigskip\noindent
{\bf Definition of Favorable Representation.} {\it
A family $\{X_\alpha: \alpha\leq \tau ,\ \alpha\ is\ isolated\}$ is a favorable representation of a space $X$ if
all of the following hold:

\begin{enumerate}
	\item $X = \bigcup_{\alpha\leq\tau}X_\alpha$;
	\item If $A\cap X_\alpha$ is functionally closed in $X_\alpha$ for all isolated $\alpha\leq \tau$, then $A$ is functionally closed in $X$;
	\item If $X_\alpha\cap X_\beta \not =\emptyset$, then $\alpha=\beta +1$ or $\beta = \alpha+1$;
	\item The union of any subcollection of $\{X_\alpha: \alpha\leq \tau ,\ \alpha\ is\ isolated\}$
is functionally closed in $X$. 
\end{enumerate}
}
\par\bigskip
A much simpler version of 
Lemma \ref{L:GOOD} for the real line is proved in \cite[Lemma 2.1]{R}. To prove our generalized version
we introduce three technical definitions that will be  used only to prove Lemma \ref{L:GOOD}.
In what follows when dealing with a favorable representation
$\{X_\alpha: \alpha\leq \tau ,\ \alpha\ is\ isolated\}$ we always assume that $X_\alpha$ is empty
if $\alpha\leq \tau$ is limit.
\par\bigskip\noindent
{\bf Definition of Odd Ordinal.} {\it
An ordinal $\alpha$ will be called odd if $|\{\beta<\alpha: |\alpha\setminus \beta |\ is\ finite\}|$ is an
odd number.
}

\par\bigskip\noindent
{\bf Definition of $\beta$-Coloring.} {\it
Let $\{Y_\alpha: \alpha\leq \tau,\ \alpha\ is\ isolated\}$ be a favorable representation of
$Y$; ${\rm dim} Y \leq n$; and $f:X\subset Y\to Y$. We say that
$\{A_i,B_i:i\leq n+1\}$ is a $\beta$-coloring of $f$ if all of the following hold:
\begin{description}
	\item[\rm C1] $\{A_i, B_i: i\leq n+1\}$ is a functionally closed cover of $\bigcup_{\gamma\leq \beta}Y_\gamma$.
	\item[\rm C2] $A_i\cap B_i=\emptyset$;
	\item[\rm C3] $f(A_i)\cap A_i =\emptyset$ and $f(B_i)\cap B_i =\emptyset$.
\end{description}
}

\par\bigskip\noindent
{\bf Definition of Agreed Colorings.} {\it
Let $\{Y_\alpha: \alpha\leq \tau,\ \alpha\ is\ isolated\}$ be a favorable representation of
$Y$; ${\rm dim} Y\leq n$; $f:X\subset Y\to Y$; and
$\{A_i^\gamma,B_i^\gamma:i\leq n+1\}$ and $\{A_i^\beta,B_i^\beta:i\leq n+1\}$ are  $\gamma$- and $\beta$-coloring of $f$
with $\gamma+2\leq \beta$. We say that the colorings agree if all of the following hold:
\begin{description}
	\item[\rm A1] If $\gamma$ and $\beta$ are odd then $A_i^\gamma\subset A_i^\beta$ and $B_i^\gamma\subset B_i^\beta$;
	\item[\rm A2] If $\beta$ is odd and $\gamma'$ is the smallest
odd ordinal greater than or equal to $\gamma$,  then $A_i^\beta\cap Y_\gamma = A_i^{\gamma'+2}\cap Y_\gamma$ and 
$B_i^\beta\cap Y_\gamma = B_i^{\gamma'+2}\cap Y_\gamma$.
\end{description}
}
\par\bigskip\noindent
Observe that in A2 of the above definition, $\gamma'+2\leq \beta$ because $\beta $ is odd and $\beta\geq \gamma +2$.
What the requirement A2 states is that $A_i^\beta$ and $A_i^\alpha$ cut out the same piece
from $Y_\gamma$ as long as $\alpha$ and $\beta$ are odd and at least two  ordinals above $\gamma$. This
property together with 2 of favorable representation will be used below to show
that the union of $A_i^\beta$'s over add $\beta$ is closed.

\par\bigskip\noindent
\begin{lem}\label{L:GOOD4LIMIT}
Let $\{Y_\alpha: \alpha\leq \tau,\ \alpha\ is\ isolated\}$ be a favorable representation of a normal space $Y$
of dimension at most $n$. Let $X$ be a closed subspace of $Y$
and $f:X\to Y$ a continuous map.
Let $\alpha$ be limit and $\{\{A_i^\beta, B_i^\beta:i\leq n+1\}: \beta<\alpha\}$ be a collection of $\beta$-colorings of
$f$ that agree with each other. Then the following sets form an $\alpha$-coloring of $f$
that agrees with all $\beta$-colorings in the given collection:
$$
A_i^\alpha=\bigcup\{B_i^\beta: \beta<\alpha, \beta\ is\ odd\},\ B_i^\alpha=\bigcup\{A_i^\beta: \beta<\alpha, \beta\ is\ odd\},\ {\rm where} \ i\leq n+1.
$$
\end{lem}
\begin{proof}
Neither A1 nor A2 in the definition of agreed colorings
needs verification since being limit $\alpha$ is not odd.
Let us verify C1-C3 in the definitions of $\alpha$-coloring.
For C1, recall that $Y_\alpha =\emptyset$. Therefore, if $x\in \bigcup_{\beta\leq \alpha}Y_\beta$
then $x\in Y_\beta$ for $\beta<\alpha$. Let $\gamma<\alpha$ be an odd ordinal greater than $\beta$.
Then by C1 for $\gamma$, we have $x\in \bigcup\{A_i^\gamma, B_i^\gamma:i\leq n+1\}$. Thus, $\{A_i^\alpha, B_i^\alpha:i\leq n+1\}$
is a cover of $\bigcup_{\beta\leq \alpha}Y_\beta$. Let us show that $A_i^\alpha$ is functionally closed in $Y$.
By property 2 of favorable representations we need to fix an isolated $\gamma$ and  show that $Y_\gamma\cap A_i^\alpha$ is functionally
closed in $Y_\gamma$. 
The inclusion $A_i^\alpha\subset \bigcup_{\beta\leq\alpha}Y_\beta$, the equality $Y_\alpha=\emptyset$, and 
3 of favorable representation imply that $A_i^\alpha$ may meet only $Y_\gamma$ with $\gamma<\alpha$.
We have $A_i^\alpha\cap Y_\gamma = \bigcup \{B_i^\beta\cap Y_\gamma: \beta<\alpha, \beta\ is\ odd\}$. This set  can be written
as the union of two sets
$T_1=\bigcup \{B_i^\beta\cap Y_\gamma: \gamma'+2\leq\beta<\alpha, \beta\ is\ odd\}$ and  
$T_2=\bigcup \{B_i^\beta\cap Y_\gamma: \beta\leq\gamma'+2, \beta\ is\ odd\}$,
where $\gamma'$ is the smallest odd ordinal greater than or equal to $\gamma$.
The set $T_1$ is $B_i^{\gamma'+2}\cap Y_\gamma$ by A2 for ordinals below $\alpha$.
The set $B_i^{\gamma'+2}$ is functionally closed by inductive assumption C1. Therefore,
$B_i^{\gamma'+2}\cap Y_\gamma$ is functionally closed in $Y_\gamma$.
The set $T_2$ is $B_i^{\gamma'+2}\cap Y_\gamma$ by A1 for ordinals below $\alpha$ and is functionally closed for the same
reasons as the first one. Thus,
$A_i^\alpha\cap Y_\gamma$ is the union of two functionally closed sets in $Y_\gamma$.

For C2, we need to show that $A_i^\beta\cap B_i^\gamma=\emptyset$ for odd $\beta,\gamma<\alpha$.
Assume $\gamma\geq \beta$. By A1 for $\gamma$, we have $A_i^\beta\subset A_i^\gamma$. By C2 for $\gamma$,
we have $A_i^\gamma\cap B_i^\gamma=\emptyset$.

For C3, we need to show that $f(B_i^\beta)\cap B_i^\gamma=\emptyset$ for odd $\beta,\gamma<\alpha$.
Assume $\gamma\geq \beta$. By A1 for $\gamma$, we have $B_i^\beta\subset B_i^\gamma$. By C3 for $\gamma$,
we have $f(B_i^\gamma)\cap B_i^\gamma=\emptyset$.
\end{proof}

\par\bigskip\noindent
\begin{lem}\label{L:GOOD}
Let $\{Y_\alpha: \alpha\leq \tau,\ \alpha\ is\ isolated\}$ be a favorable representation of a normal space $Y$
of dimension at most $n$. Let $X$ be a closed subspace of $Y$
and $f:X\to Y$ a fixed-point free continuous map with the following properties:
\begin{enumerate}
	\item $X$ misses $Y_0$; and
	\item If $z\in Y_\alpha\cap X$ then $f(z)\in [\bigcup_{\beta<\alpha} Y_\beta]\setminus Y_\alpha$.
\end{enumerate}
Then there exists a $(2n+2)$-sized coloring $\mathcal F$ of $f$  such that $cl_{Y}(f(F))\cap F=\emptyset$
for every $F\in {\mathcal F}$.
\end{lem}
\begin{proof}
Inductively, we will construct building blocks for our colors.

\par\bigskip\noindent
\underline {\it Step 0}: Put $A_1^0=Y_0$, $A_i^0 = \emptyset$ if $1<i\leq n+1$, and $B_i^0=\emptyset$ if $i\leq n+1$.

\par\bigskip\noindent
\underline {\it Induction Assumption}: Assume that for $\beta <\alpha$ we have defined 
families $\{A_i^\beta, B_i^\beta: i\leq n+1\}$ that are $\beta$-colorings of $f$ and agree with each other.

\par\bigskip\noindent
Observe that for $\beta = 0$ the family $\{A_i^0,B_i^0:i\leq n+1\}$ meets C1-C3, A1, and A2
in the definitions of $\beta$-coloring and agreed colorings. Indeed, C1  holds because $A_1^0=Y_0$.
C2  holds because $B_i^0=\emptyset$. C3  holds because $f$ is not defined on $Y_0$. A1 and A2 are not applicable for $\beta=0$
since $0$ is not odd.

\par\bigskip\noindent
\underline {{\it Step} [$\alpha = \lim\{ \beta < \alpha\}$]}: We construct $\alpha$-coloring
as in Lemma \ref{L:GOOD4LIMIT}.  
\par\bigskip\noindent 
\underline {{\it Step} [$\alpha=\beta+1$]}: 
Put
$$
C_i^\alpha=[f^{-1}(A_i^\beta)\cap Y_\alpha]\cup B_i^\beta\ {\rm and}\ D_i^\alpha=[f^{-1}(B_i^\beta)\cap Y_\alpha]\cup A_i^\beta .
$$
Let us show that these sets have the following properties:
\begin{itemize}
	\item[S1.]  $C_i^\alpha$ and $D_i^\alpha$ are functionally closed;
	\item[S2.]  $C_i^\alpha\cap D_i^\alpha=\emptyset$; and
	\item[S3.]  If $x\in \bigcup_{\gamma\leq \alpha}Y_\gamma\cap X$ then both $f(x)$ and $x$ are in $\bigcup\{C_i^\alpha,D_i^\alpha:i\leq n+1\}$. 
\end{itemize}
For S1, observe that $f^{-1}(A_i^\beta)$ is functionally closed as the inverse image of the set functionally closed by C1 for $\beta$.
By property 4 of favorable representation, $f^{-1}(A_i^\beta)\cap Y_\alpha$ is functionally closed in $Y$ as well.
The set $B_i^\beta$ is functionally closed by C1  for $\beta$. Thus, $C_i^\alpha$ is functionally closed.
For S2 we need to verify the following four equalities:
\begin{enumerate}
	\item $[f^{-1}(A_i^\beta)\cap Y_\alpha]\cap [f^{-1}(B_i^\beta)\cap Y_\alpha]=\emptyset$;
	\item $[f^{-1}(A_i^\beta)\cap Y_\alpha]\cap  A_i^\beta=\emptyset$;
	\item $B_i^\beta\cap [f^{-1}(B_i^\beta)\cap Y_\alpha]=\emptyset$; and
	\item $B_i^\beta\cap A_i^\beta=\emptyset$.
\end{enumerate}
Equality 1 holds because $A_i^\beta\cap B_i^\beta = \emptyset$ by C2  for $\beta$.
Equalities  2 and 3 hold by C3  for $\beta$.  
Equality 4 holds by C2  for $\beta$.
Let us show S3. 
By 2 of the lemma hypothesis, $f(x)\in \bigcup_{\gamma\leq \beta} Y_\gamma$
and $x\in [f^{-1}(\bigcup_{\gamma\leq \beta} Y_\gamma)\cap Y_\alpha]\cup [\bigcup_{\gamma\leq \beta} Y_\gamma]$.
By C1  for $\beta$, 
$\bigcup_{\gamma\leq \beta} Y_\gamma = \bigcup\{A_i^\beta,B_i^\beta:i\leq n+1\}$ and 
$[f^{-1}(\bigcup_{\gamma\leq \beta} Y_\gamma)\cap Y_\alpha] \subset \bigcup \{f^{-1}(A_i^\beta)\cap Y_\alpha,f^{-1}(B_i^\beta)\cap Y_\alpha:i\leq n+1\}$.
The right-hand sets in these formulas are subsets of $\bigcup \{C_i^\alpha, D_i^\alpha:i\leq n+1\}$
by the definition.

Since $C_i^\alpha$ and $D_i^\alpha$ are functionally closed and disjoint, by Lemma \ref{L:ENLARGE}, there exists
a family $\{A_i^\alpha,B_i^\alpha : i\leq n+1\}$ of subsets of $\bigcup_{\gamma\leq \alpha}Y_\gamma$ with the following properties:
\begin{itemize}
	\item[S4.] $\{A_i^\alpha,B_i^\alpha: i\leq n+1\}$  is a functionally closed cover of $\bigcup_{\gamma\leq \alpha}Y_\gamma$;
	\item[S5.] $A_i^\alpha\cap Z =C_i^\alpha$ and $B_i^\alpha\cap Z=D_i^\alpha$, where $Z=\bigcup\{C_j^\alpha,D_j^\alpha: j\leq n+1\}$;
	\item[S6.] $A_i^\alpha\cap B_i^\alpha=\emptyset$. 
\end{itemize}
In addition, if $x\in A_i^\alpha\cap X$ then, by S3, $x\in Z = \bigcup\{C_i^\alpha, D_i^\alpha:i\leq n+1\}$.
Thus, $x\in A_i^\alpha\cap Z$. By S5, $x\in C_i^\alpha$. In short,

\begin{itemize}
	\item[S7.] If $x\in A_i^\alpha\cap X$ ($B_i^\alpha\cap X$), then $x\in C_i^\alpha$ ($D_i^\alpha$).
\end{itemize}

Also, if $x\in A_i^\alpha\setminus C_i^\alpha$ then, by S5, $x\not \in Z$.
The set $Z$ contains
$\bigcup_{\gamma\leq \beta}Y_\gamma$ due to second summands in the definitions
of $C_i^\alpha$ and $D_i^\alpha$ and assumption C1  for $\beta$. 
Therefore, $x\not\in \bigcup_{\gamma\leq \beta}Y_\gamma$.
Hence, $x\in Y_\alpha$. In summary,

\begin{itemize}
	\item[S8.] 
If $x\in A_i^\alpha\setminus C_i^\alpha$ then $x\in Y_\alpha$.
\end{itemize}

Let us check C1-C3, A1, and A2. Property C1  is S4 and C2  is S6. Let us demonstrate C3  for $A_i^\alpha$.
By S7, $f(A_i^\alpha)\cap A_i^\alpha=f(C_i^\alpha)\cap A_i^\alpha$. 
By S3, $f(C_i^\alpha)\subset Z=\bigcup \{C_i^\alpha, D_i^\alpha:i\leq n+1\}$. By S5, $A_i^\alpha\cap Z = C_i^\alpha$.
Therefore $f(C_i^\alpha)\cap A_i^\alpha = f(C_i^\alpha)\cap C_i^\alpha$.
To show that the last intersection is empty
we need to verify the following equalities.
\begin{enumerate}
	\item $f(f^{-1}(A_i^\beta)\cap Y_\alpha)\cap [f^{-1}(A_i^\beta)\cap Y_\alpha]=\emptyset$;
	\item $f(f^{-1}(A_i^\beta)\cap Y_\alpha)\cap B_i^\beta=\emptyset$;
	\item $f(B_i^\beta)\cap [f^{-1}(A_i^\beta)\cap Y_\alpha]=\emptyset$; and
	\item $f(B_i^\beta)\cap B_i^\beta=\emptyset$.
\end{enumerate}
Equality 1 holds by C3  for $\beta$. Equality 2 holds by C2  for $\beta$. Let us show 3. 
If $\beta = 0$ then, by 1 of the lemma's hypothesis, $f(B_i^\beta)$ is empty and 3 holds. Now assume $\beta>0$. From
2 of the lemma's hypothesis and emptiness of $Y_\lambda$ for limit $\lambda$, we
conclude that the first set is in $\bigcup_{\gamma<\beta} Y_\gamma$. The second set is
in $Y_\alpha$ by the definition. By 3 of favorable representation, $\bigcup_{\gamma<\beta} Y_\gamma$
misses $Y_\alpha$. Hence the intersection is empty. Equality 4 is C3  for $\beta$.

To verify A1 we  assume $\alpha$ is odd and pick an odd $\gamma<\alpha$. By construction,
$A_i^\alpha$ contains $[f^{-1}(A_i^\beta)\cap Y_\alpha]\cup B_i^\beta$. It suffices to show that $B_i^\beta$ contains $A_i^\gamma$.
If $\beta$ is limit then, by construction, $B_i^\beta = \bigcup\{A_i^\lambda: \lambda<\beta,\ \lambda\ is\ odd\}$ and the right side contains 
$A_i^\gamma$. If $\beta = \lambda +1$ then, by construction,  $B_i^\beta$ contains  $[f^{-1}(B_i^\lambda)\cap Y_\beta]\cup A_i^\lambda$. By A1 for $\lambda$,
 $A_i^\lambda$ contains $A_i^\gamma$.

Finally to show $A2$ we assume $\alpha$ is odd and pick a $\gamma$ such that $\gamma'+2<\alpha$, where $\gamma'$
is the smallest odd ordinal greater than or equal to $\gamma$.  
Observe that $A_i^\alpha=C_i^\alpha\cup (A_i^\alpha\setminus C_i^\alpha)$
which in its turn equals $[f^{-1}(A_i^\beta)\cap Y_\alpha]\cup B_i^\beta\cup  (A_i^\alpha\setminus C_i^\alpha)$.
The first set in this union is in $Y_\alpha$ and so is the third by S8.  By
3 of favorable representation $Y_\alpha$ misses $Y_\gamma$. Thus,
$A_i^\alpha\cap Y_\gamma = B_i^\beta\cap Y_\gamma$.

If $\beta$ is limit, then $B_i^\beta\cap Y_\gamma = \bigcup\{A_i^\lambda\cap Y_\gamma: \lambda<\beta,\ \lambda\ is \ odd\}$.
The right set can be written as the union of $T_1= \bigcup\{A_i^\lambda\cap Y_\gamma: \lambda\leq\gamma'+2,\ \lambda\ is \ odd\}$
and $T_2=\bigcup\{A_i^\lambda\cap Y_\gamma: \gamma'+2\leq\lambda<\beta,\ \lambda\ is \ odd\}$. By A1, $T_1 = A_i^{\gamma'+2}\cap Y_\gamma$.
By A2, $T_2=A_i^{\gamma'+2}\cap Y_\gamma$.

Now assume $\beta$ is isolated. Since there exists odd $\gamma'$ below odd $\alpha$ we conclude that $\beta\not=0$.
Therefore, $\beta = \lambda+1$. We have $B_i^\beta =D_i^\beta\cup (B_i^\beta\setminus D_i^\beta)$,
which in its turn equals the union of three sets $T_1=[f^{-1}(B_i^\lambda)\cap Y_\beta]$, 
$T_2=A_i^\lambda$,  and $T_3= (B_i^\beta\setminus D_i^\beta)$. 
Since $\alpha>\gamma+2$, we conclude $\beta>\gamma+1$.
By property 3 of favorable representations, $Y_\beta$ misses $Y_\gamma$. 
We have $T_1$ is in $Y_\beta$ and so is $T_3$ by S8 for isolated $\beta$.
Therefore, only $T_2$ can meet $Y_\gamma$.
Hence, $B_i^\beta\cap Y_\gamma = A_i^\lambda\cap Y_\gamma$.
By A2 for $\lambda$, the right side of the last equality is $A_i^{\gamma'+2}\cap Y_\gamma$.
Inductive construction is complete.

\par\bigskip
Put $A_i = A_i^\tau$ and $B_i = B_i^\tau$. By C1-C3 , the family ${\mathcal G}=\{A_i, B_i:i\leq n+1\}$ 
is a closed cover of $Y$ such that $f(G)\cap G=\emptyset$ for every $G\in {\mathcal G}$.
By 3 of Lemma \ref{L:COMMONFACTS}, ${\mathcal F}=\{f^{-1}(G):G\in {\mathcal G}\}$ is a desired coloring of $f$.
The lemma is proved.
\end{proof}


\par\bigskip\noindent
\begin{lem}\label{L:COMPACT}
Let $X$ be a compact subspace of ${\mathbb R}^n$ and $f:X\to {\mathbb R}^n$ a continuous fixed-point free map.
Then $f$ is colorable in at most $4n(n+1)$ colors.
\end{lem}
\begin{proof}
For each $i\leq n$, put $A_i= \{x\in X: \pi_i\circ f(x)>\pi_i(x)\}$ and
$B_i= \{x\in X: \pi_i\circ f(x)<\pi_i(x)\}$.
By continuity of $f$, each of these sets is open. Since $f$ is fixed-point free,
$\{A_i,B_i:i\leq n\}$ is a cover of $X$. By paracompactness of $X$, 
there exists  a closed shrinking $\{A'_i, B_i':i\leq n\}$. It suffices to show
now that $f|_S$ is colorable in $2(n+1)$ colors for each $S\in \{A_i',B_i':i\leq n\}$.
We will demonstrate it for $A= A_1'$.

Fix $a\in A$. Since $A$ is compact and $\pi_1\circ f(x) >\pi_1(x)$ for every  $x\in A$ 
we can find $\epsilon_a>0$ such that

\par\bigskip\noindent
(*) If $x\in A$ and $\pi_1(x)\in [\pi_1(a)-\epsilon_a, \pi_1(a)+\epsilon_a]$ then $\pi_1\circ f(x)>\pi_1(a)+ \epsilon_a$.

\par\bigskip\noindent
By compactness and (*) we can select a strictly decreasing sequence of reals $r_1>....>r_m$
such that 
\begin{itemize}
 	\item[P1.] $A\subset [r_m,r_1)\times {\mathbb R}^{n-1}$; and
	\item[P2.] If $x\in [r_{k+1}, r_k]\times {\mathbb R}^{n-1}\cap A$ then $\pi_1\circ f(x)> r_k$.
\end{itemize}

Put $X_0 = [r_1, \infty)\times {\mathbb R}^{n-1}$; $X_k = [r_{k+1}, r_k]\times R^{n-1}$ if $1\leq k\leq m-1$; and
$X_m = (-\infty,r_m]\times {\mathbb R}^{n-1}$. Clearly $\{X_k\}_{k\leq m}$ is a favorable representation of
${\mathbb R}^n$. We only need to show that ${\mathbb R}^n$, $\{X_k\}_{k\leq m}$, $f|_A$, and $A$ meet 1 and 2 of Lemma \ref{L:GOOD}.
Requirement 1 of Lemma \ref{L:GOOD} is met because $X_0$ misses $A$ by P1. Requirement 2 of Lemma \ref{L:GOOD} is met in P2.
The lemma is proved. 
\end{proof}

\par\bigskip\noindent
\begin{thm}\label{T:EUCLIDEAN}
Let $X$ be a closed subspace of ${\mathbb R}^n$ and $f:X\to {\mathbb R}^n$
a continuous fixed-point free map.
Then $f$ is colorable.
\end{thm}
\begin{proof}
Put $X_1 =[-1,1]^n$ and $X_{k+1}=[-(k+1), (k+1)]^n\setminus (-k,k)^n$. Let $M= 4n(n+1)$.

\par\smallskip\noindent
{\it Claim 1.} {\it
Let $A = \{x\in X: x\in X_k\ and\ f(x)\in X_m,\ where\ m\geq k\}$. Then there exists
a finite open cover ${\mathcal U}_A$ of $A$ such that the closure of every element
of ${\mathcal U}_A$ is a color.
}
\par\smallskip\noindent
{\it Proof of Claim 1.} Clearly, $A$ is closed. Fix a strictly increasing sequence $\langle b_k\rangle_k$ of positive integers
that have the following property.
\par\smallskip\noindent
\begin{itemize}
	\item[P1.] $f([-b_k,b_k]^n\cap X)\subset (-b_{k+1},b_{k+1})^n$.
\end{itemize}
\par\smallskip\noindent
For every $k$, select ${\mathcal S}_k = \{S_1^k,...,S_M^k\}$ a coloring of $f_k$, where
$f_k$ is the restriction of $f$ to
$A\cap ([-b_k,b_k]^n\setminus (-b_{k-1}, b_{k-1})^n)$.
This is possible by Lemma \ref{L:COMPACT}. 
By the definition of $A$, we have  $f(S)$ does not meet $(-b_{k-1},b_{k-1})^n$ for any $S\in {\mathcal S}_k$. 
Also $f(S)$ does not meet the compliment of $(b_{k+1}, b_{k+1})^n$ by P1.
In short, we have
\par\smallskip\noindent
\begin{itemize}
	\item[P2.] $f(S)\subset (-b_{k+1},b_{k+1})^n\setminus (-b_{k-1},b_{k-1})^n$ for every $S\in {\mathcal S}_k$.
\end{itemize}
\par\smallskip\noindent
Put $O_i = \bigcup \{S_i^k:k\ is\ odd\}$ and $E_i = \bigcup \{S_i^k: k\ is\ even\}$.
The sets $O_i$ and $E_i$ are closed as unions of discrete families of closed sets
and $\{O_i, E_i :i\leq M\}$ is a cover of $A$.
Let us show that $f(O_1)\cap O_1=\emptyset$. By P2, we have $f(S_1^k)\cap S_1^m=\emptyset$ if
$k\not = m$. Since $S_1^k$ is a color, $f(S_1^k)\cap S_1^k =\emptyset$.
Thus we have colored $f|_A$ in  $2M= 8n(n+1)$ many colors. 
By 1 of Lemma \ref{L:COMMONFACTS}, to find a desired open cover it suffices to show that $f(O_i)$ and $f(E_i)$ are
closed. We have $f(O_i)=\bigcup\{f(S_i^k): i\ is\ odd\}$. Each $f(S_i^k)$ is compact since $S_i^k$ is.
By P2, $\{f(S_i^k): i\ is\ odd\}$ is a locally finite family. Therefore, $f(O_i)$ is closed.
The claim is proved.

\par\bigskip\noindent
{\it Claim 2.} {\it
Let $B =X\setminus \bigcup {\mathcal U}_A$. Then there exists a finite cover ${\mathcal G}_B$ of $B$ by colors.
}
\par\smallskip\noindent
{\it Proof of Claim 2.}
Since $B$ misses $A$ the following holds.
\par\smallskip\noindent
(*) If $x\in B\cap X_k$ and $f(x)\in X_m$ then $m<k$.
\par\smallskip\noindent
Since $\{X_k\}_k$ is a favorable representation of ${\mathbb R}^n$ we only need to show that 1 and 2 of Lemma \ref{L:GOOD} are met.
Since $X_1$ is a subset of $A$ it misses $B$, so 1 of Lemma \ref{L:GOOD} is met. 
Requirement 2 of Lemma \ref{L:GOOD} is given by (*). 
The claim is proved.
\par\medskip
Clearly ${\mathcal G}_B\cup \{\bar U: U\in {\mathcal U}_A\}$ is a finite coloring of $f$.
\end{proof}


\subsection{Chromatic number}\label{SS:chromatic}
Given a fixed-point free map $f:X\to Y$, the {\it chromatic number} of $f$ is the smallest size of a coloring of $f$.
First of all we record the following statement for future references. It extends the main result of \cite{R}.

\par\bigskip\noindent
\begin{cor}\label{R2N2R2N}
Every continuous fixed-point free self-map of ${\mathbb R}^n$ is colorable in at most $n+3$ colors.
\end{cor}
\begin{proof}
In \cite[Theorem 3.12.17]{VM2}
the following theorem (due to R. Pol) is proved:
{\it Let $X$ be a separable metrizable space and let
$f:X\to X$ be a fixed-point free continuous map.
If $f$ is finitely colorable and $\dim X\leq n$ then
$f$ can be colored with $n+3$ colors.} 
This theorem together with Theorem \ref{T:EUCLIDEAN} give the desired conclusion. 
\end{proof}

\par\bigskip
We would like to remark that the theorem cited in the proof of Corollary \ref{R2N2R2N} holds
if one replace "separable metrizable" by "normal". This fact is observed in \cite{VM} and is attributed 
to Pol. Fur further reference, Theorem 3.12.17 of \cite{VM2} will be referred to as the Pol theorem.
\par\bigskip
Next we estimate the chromatic number of continuous fixed-point free maps $f \colon X \to {\mathbb R}^{n}$, 
thus extending Corollary \ref{R2N2R2N} and making Theorem \ref{T:EUCLIDEAN} more precise. 
For this we need the following observation.

\begin{lem}\label{L:emb+1}
Let $n \geq 0$ and $f \colon X \to Y$ be a continuous fixed-point free map defined on a closed subspace 
$X$ of a metrizable $n$-dimensional $AE(n)$-space $Y$. Then there exist an embedding 
$X \subset Y \times [0,\infty)$ and a continuous fixed-point free map 
$g \colon Y \times [0,\infty) \to Y \times [0,\infty)$ such that $g|X = f$.
\end{lem}
\begin{proof}
Let $h \colon Y \to Y$ be a continuous  extension of $f$. Such $h$ exists since 
$\dim Y =n$ and $Y \in AE(n)$. Let $\operatorname{Fix}(h) = \{ y \in Y \colon h(y) = y\}$. 
Since $f$ is fixed-point free and $h|X = f$ it follows that $X \cap \operatorname{Fix}(h) = \emptyset$. 
Identify $Y$ with the subspace $Y \times \{ 0\}$ of the product $Y \times [0,\infty)$, $X$ with the 
subspace $X \times \{ 0\}$ and define $g$ as follows:
\[ g(y,r) = \left(h(y), r + \operatorname{dist}(y, X)\right). \]
 
For $(x,0) \in X \times \{0\}$ we have $g(x,0) = \left(h(x), \operatorname{dist}(x, X)\right) = (f(x),0)$ 
which shows that $g$ is an extension of $f$. To see that $g$ is fixed-point free note that 
\begin{enumerate}
\item 
if $y \in \operatorname{Fix}(h)$, then $y\not\in X$, and thus $\operatorname{dist}(y,X) >0$ and $g(y,r) \not= (y,r)$;
\item 
if $y \not\in \operatorname{Fix}(h)$, then $h(y) \not= y$ and thus $g(y,r) \not= (y,r)$. 
\end{enumerate}
\end{proof}

\begin{pro}\label{P:chromatic}
Every continuous fixed-point free map $f \colon X \to Y$, defined on a closed subset $X$ of an at most $n$-dimensional locally 
compact separable metrizable space $Y$ (e.g. $Y = {\mathbb R}^{n}$), is colorable in at most $n+3$ colors.
In addition, there exists an $(n+3)$-sized coloring $\mathcal F$ of $f$ such that
$F\cap cl_Y(f(F))=\emptyset$ for every $F\in {\mathcal F}$.
\end{pro}
\begin{proof}
Without loss of generality we may assume that $Y$ is a closed subspace of ${\mathbb R}^{2n+1}$. 
By Lemma \ref{L:emb+1} and Corollary \ref{R2N2R2N}, there is a colorable (in at most $2n+5$ colors) 
map $g \colon {\mathbb R}^{2n+2} \to {\mathbb R}^{2n+2}$ such $g|X = f$. Let $p \colon N_{n}^{2n+1} \to {\mathbb R}^{2n+2}$ 
be an $n$-soft map of the $n$-dimensional universal N\"{o}belling space (see \cite[Theorem 5.6.4]{chibook}). 
Since $\dim Y \leq n$ and $p$ is $n$-soft, there exists a map $j\colon Y \to N_{n}^{2n+1}$ such that 
$p\circ j = \operatorname{id}_{Y}$. Clearly $j(X)$ and $j(Y)$ are copies of $X$ and $Y$ in $N_{n}^{2n+1}$ 
and the map $f^{\prime} \colon j(X) \to j(Y)$, defined as $f^{\prime} = j\circ f\circ p|j(X)$, is a copy of $f$. 
Consequently, it suffices to prove the required upper bound for the chromatic number of $f^{\prime}$. 
Using $n$-softness of $p$ we conclude that there 
exists a map $h \colon N_{n}^{2n+1} \to N_{n}^{2n+1}$ such that $p\circ h = g\circ p$ and 
$h|j(X) = j\circ g\circ p|j(X) = j\circ f\circ p|j(X) =f^{\prime}$. Colorability of $g$ and the equality 
$p\circ h = g\circ p$ imply that $h$ is also colorable. 
Since the dimension of $N_n^{2n+1}$ is $n$, by the Pol theorem \cite[Theorem 3.12.17]{VM2}, $h$ is colorable in at most $n+3$ colors. 
Since $h$ is a self-map, by 3 of Lemma \ref{L:COMMONFACTS}, we can find $\mathcal F$ with the required properties.
\end{proof}



\section{Fixed-point free maps of locally compact paracompact spaces}\label{S:paracompact}

\begin{pro}\label{P:lc}
Every continuous fixed-point free map $f \colon X \to Y$, defined on a closed subspace of an at most $n$-dimensional locally 
compact and Lindel\"{o}f space $Y$, is colorable in at most $n+3$ colors. 
In addition, there exists an $(n+3)$-sized coloring $\mathcal F$ of $f$ such that
$F\cap cl_Y(f(F))=\emptyset$ for every $F\in {\mathcal F}$.
\end{pro}
\begin{proof}
By Proposition \ref{P:chromatic}, we may assume that 
$\omega(Y) > \omega$. First we need the following observation.
\par\medskip\noindent
{\it Claim}. {\it $Y = \lim{\mathcal S}_{Y}$, where ${\mathcal S}_{Y} = \{ Y_{\alpha}, q_{\alpha}^{\beta},A\}$ 
is a factorizing $\omega$-spectrum consisting of locally compact separable metrizable spaces $Y_{\alpha}$ and 
proper projections $p_{\alpha}^{\beta} \colon Y_{\beta} \to Y_{\alpha}$, $\beta \geq \alpha$.}
\par\smallskip\noindent
{\it Proof of Claim}. Let $\widetilde{Y}$ be the one-point compactification of $Y$. Represent $\widetilde{Y}$ 
as the limit of a factorizing $\omega$-spectrum ${\mathcal S}_{\widetilde{Y}} = \{ \widetilde{Y}_{\alpha}, \widetilde{q}_{\alpha}^{\beta}, B\}$, 
consisting of metrizable compact spaces $\widetilde{Y}_{\alpha}$ and surjective projections. Since $Y$ is locally compact and 
Lindel\"{o}f, it follows that $Y$ is functionally open in $\widetilde{Y}$. Since ${\mathcal S}_{\widetilde{Y}}$ is a 
factorizing spectrum, there exists an index $\alpha_{0} \in B$ such that $Y = \widetilde{q}_{\alpha_{0}}^{-1}(\widetilde{q}_{\alpha_{0}}(Y))$. 
Let $A = \{ \alpha \in B \colon \alpha \geq \alpha_{0}\}$, $Y_{\alpha} = \widetilde{q}_{\alpha}(Y)$, and 
$q_{\alpha}^{\beta} = \widetilde{q}_{\alpha}^{\beta}|Y_{\beta}$, $\beta \geq \alpha$, $\alpha, \beta \in A$. 
Then the limit of the spectrum ${\mathcal S}_{Y} = \{ Y_{\alpha}, q_{\alpha}^{\beta}, A\}$ coincides with $Y$. 
Moreover, each $Y_{\alpha}$, $\alpha \in A$, is locally compact (as an open subspace of $\widetilde{Y}_{\alpha}$ and 
each $q_{\alpha}^{\beta} \colon Y_{\beta} \to Y_{\alpha}$ is proper (as the restriction of $\widetilde{q}_{\alpha}^{\beta}$ 
onto the inverse image $Y_{\beta} = (\widetilde{q}_{\alpha}^{\beta})^{-1}(Y_{\alpha})$). Finally, by \cite[Corollary 1.3.2]{chibook}, 
${\mathcal S}_{Y}$ is a factorizing $\omega$-spectrum. This proves our Claim.

\par\medskip\noindent
Since $X$ is closed in $Y$ it is the limit of the induced spectrum ${\mathcal S}_{X} = \{ X_{\alpha}, p_{\alpha}^{\beta}, A\}$, 
where $X_{\alpha} = q_{\alpha}(X)$ and $p_{\alpha}^{\beta} = q_{\alpha}^{\beta}|X_{\beta}$, 
$\beta \geq \alpha$, $\alpha, \beta \in A$. Note that ${\mathcal S}_{X}$ is also an $\omega$-spectrum. 
It is factorizing since so is ${\mathcal S}_{Y}$ and $X$ is $C$-embedded in $Y$.

By \cite[Theorems 1.3.4 and 1.3.10]{chibook}, we may assume without loss of generality that each $Y_{\alpha}$ in the 
spectrum ${\mathcal S}_{Y}$ is at most $n$-dimensional. Again by the Spectral Theorem \cite[Theorem 1.3.4]{chibook}, 
we may assume (if necessary passing to a cofinal and $\omega$-complete subset of $A$) that for each $\alpha \in A$ there 
is a map $f_{\alpha} \colon X_{\alpha} \to Y_{\alpha}$ such that $q_{\alpha}\circ f = f_{\alpha}\circ p_{\alpha}$. Since $f$ is 
fixed-point free and $X$ is Lindel\"{o}f, we can find a countable functionally open cover $\{ G_{i} \colon i \in \omega\}$ of $X$ and 
a countable collection $\{ U_{i} \colon \in \omega\}$ of functionally open subsets of $Y$ such that $f(G_{i}) \subset U_{i}$ and 
$U_{i}\cap G_{i} = \emptyset$ for each $i \in \omega$. Factorizability of spectra ${\mathcal S}_{X}$ and ${\mathcal S}_{Y}$ 
guarantees (see \cite[Proposition 1.3.1]{chibook}) the existence of an index $\alpha_{i} \in A$ such that 
$G_{i} = p_{\alpha_{i}}^{-1}(p_{\alpha_{i}}(G_{i}))$ and $U_{i} = (q_{\alpha_{i}})^{-1}(q_{\alpha_{i}}(U_{i}))$, $i \in \omega$. 
Choose $\beta \in A$ so that $\beta \geq \alpha_{i}$ for each $i \in \omega$  -- this is possible because $A$ 
is an $\omega$-complete set (see \cite[Corollary 1.1.28]{chibook}) --  and note that $G_{i} = p_{\beta}^{-1}(p_{\beta}(G_{i}))$ and 
$U_{i} = (q_{\beta})^{-1}(q_{\beta}(U_{i}))$ for each $i \in \omega$. Obviously, $p_{\beta}(G_{i}) \cap q_{\beta}(U_{i}) = \emptyset$, 
$i \in \omega$. We claim that $f_{\beta} \colon X_{\beta} \to Y_{\beta}$ is fixed-point free. Assuming the opposite, pick a point 
$x \in X_{\beta}$ with $f_{\beta}(x) = x$, choose $y \in X$ so that $p_{\beta}(y) = x$ and pick an index $i \in \omega$ such that 
$y \in G_{i} = p_{\beta}^{-1}(p_{\beta}(G_{i}))$. Note that $f(y) \in U_{i}$ and $q_{\beta}(f(y)) \in q_{\beta}(U_{i})$. But 
$q_{\beta}(f(y)) = f_{\beta}(p_{\beta}(y)) = f_{\beta}(x) = x \in p_{\beta}(G_{i})$ which is impossible since 
$p_{\beta}(G_{i}) \cap q_{\beta}(U_{i}) = \emptyset$.

By Proposition \ref{P:chromatic}, $f_{\beta}$ is colorable in $n+3$ colors and there exists a closed cover 
$\{ F_{1},\dots, F_{n+3}\}$ of $X_{\beta}$ such that $cl_{Y_\beta}(f_{\beta}(F_{j})) \cap F_{j} = \emptyset$ for each $j = 1,\dots,n+3$. 
It only remains to note that $\{ p_{\beta}^{-1}(F_{1}),\dots,p_{\beta}^{-1}(F_{n+3})\}$ is a closed cover of 
$X$ and $f(p_{\beta}^{-1}(F_{j})) \cap p_{\beta}^{-1}(F_{j}) =\emptyset$ for each $j = 1,\dots, n+3$.
\end{proof}

\par\bigskip\noindent
\begin{lem}\label{LEMMA:PARACOMPACT1}
Let $X$ be a closed subspace of a locally compact paracompact space $Y$ and $f:X\to Y$ continuous. Then
$Y$ can be written as $\oplus \{Y_\alpha: \alpha\ is \ isolated, \alpha \leq\tau\}$, where
$Y_\alpha$ is Lindel\"of and $f(X\cap Y_\alpha)\subset \bigcup_{\beta\leq \alpha } Y_\beta$.
\end{lem}

\begin{proof}
By Theorem 5.1.27 in  \cite{Eng}, $Y$ can be written as $\oplus_{\alpha<\kappa}Z_\alpha$, where each $Z_\alpha$ is
Lindel\"of. Suppose for every isolated $\beta<\alpha$, where $\alpha$ is isolated, $Y_\beta$ is defined and the following hold:
\begin{enumerate}
	\item If $Y_\beta$ meets $Z_\gamma$ then $Y_\beta$ contains $Z_\gamma$;
	\item $f(Y_\beta)\subset \bigcup_{\gamma\leq \beta}Y_\gamma$; 
	\item $Y_\beta$ is clopen in $Y$; and
\item $Y_\beta $ is Lindel\"of.
\end{enumerate}
If $\bigcup_{\beta<\alpha}Y_\beta$ covers all of $Y$ then 
put $\tau =\alpha$ if $\alpha$ is limit and $\tau =\lambda$ if $\alpha=\lambda+1$ (notice that $\alpha$ cannot be $0$
if $Y$ is not empty).
The family
$\oplus \{Y_\beta: \beta\leq \tau, \beta\ is\ isolated\}$ is a desired sum
by properties 1-4.
Otherwise, we define the next piece as follows.

\par\bigskip\noindent
{\it Construction of $Y_\alpha$}: Let $\lambda_0$ be an ordinal such that $Z_{\lambda_0}$ does not meet
$\bigcup_{\beta<\alpha} Y_\beta$. Put $S_0 =\{\lambda_0\}$ and, recursively, 
$S_{n+1} = \{\lambda: Z_\lambda\ meets f(Z_\gamma), \ \gamma\in S_n\}$.
Since $Z_\gamma$ is Lindel\"of for  all $\gamma$, $f(Z_\gamma)$ can meet only countably many summands in our original free sum.
Therefore, $S_n$ is countable.
Put $Y_\alpha =\bigcup\{Z_\gamma: \gamma\in \bigcup_n S_n,\ Z_\gamma\ misses \ \bigcup_{\beta<\alpha}Y_\beta\}$.
Inductive requirements 1 and 2 are explicitly incorporated in the construction. Requirement 3 is met because
$Y_\alpha$ is the union of a discrete subfamily of clopen sets. Requirement 4 follows from the Lindel\"of
property of every $Z_\gamma$ and countability of $S_n$.
 \end{proof}

\par\bigskip\noindent
\begin{lem}\label{LEMMA:PARACOMPACT2}
Let $Y=\oplus\{Y_\alpha:\alpha\leq\tau,\ \alpha\ is \ isolated\}$, where
each $Y_\alpha$ is locally compact and Lindel\"of; $\dim Y\leq n$; and $X$ closed in $Y$.
Suppose $f:X\to Y$ is continuous and the following hold:

\par\smallskip\noindent
(*) $X$ misses $Y_0$; and 
\par\smallskip\noindent
(**) $f(X\cap Y_\alpha)\subset \bigcup_{\beta<\alpha}Y_\beta$ if $\alpha\not = 0$.

\par\smallskip\noindent
Then there exists a coloring $\mathcal F$ of $f$ such that $cl_Y(f(F))\cap F = \emptyset$
for every $F\in {\mathcal F}$.
\end{lem}
\begin{proof}
Clearly, the free sum in the hypothesis is a favorable representation of $Y$.
Let us show that the conditions of Lemma \ref{L:GOOD} are met.
Requirement 1 of Lemma \ref{L:GOOD} is met by (*). Requirement 2 of Lemma \ref{L:GOOD}
 is met by (**) and the fact that $Y_\alpha$ misses $\bigcup_{\beta<\alpha}Y_\beta$,
which is thanks to our free sum representation.
\end{proof}

\par\bigskip\noindent
\begin{thm}\label{THEOREM:PARACOMPACT}
Any fixed-point free map $f \colon X \to Y$, defined on a closed subspace $X$ of a locally compact and 
paracompact space $Y$ of dimension $\dim Y \leq n$, is colorable.
In addition, there exists a coloring $\mathcal F$ of $f$ such that $cl_Y(f(F))\cap F=\emptyset$
for every $F\in {\mathcal F}$.
\end{thm}

\begin{proof}
Represent $Y$ as $\oplus\{Y_\alpha: \alpha\leq \tau,\ \alpha\ is\ isolated\}$ as 
in the conclusion of Lemma \ref{LEMMA:PARACOMPACT1}.
For each isolated $\alpha\leq \tau$ not equal $0$, put $Z_\alpha = f^{-1}(\oplus_{\beta<\alpha} Y_\beta)\cap Y_\alpha$.
Observe that each $Z_\alpha$ is clopen in $X$ being the inverse image of 
the clopen set $\oplus_{\beta<\alpha} Y_\beta$ intersected with the clopen set $Y_\alpha$.
Put $Z=\oplus_{0<\alpha\leq \tau} Z_\alpha$. The triple $\{Z, f|_Z, Y\}$ satisfies the hypothesis
of Lemma \ref{LEMMA:PARACOMPACT2}. Therefore we can find a finite cover ${\mathcal F}_Z$ of $Z$ by colors
such that $cl_Y(f(F))\cap F=\emptyset$ for every $F\in {\mathcal F}_Z$.

For each isolated $\alpha\leq \tau$, put $P_\alpha = (X\cap Y_\alpha)\setminus Z_\alpha$. Since $Z_\alpha$ is open,
$P_\alpha$ is closed in $Y_\alpha$. Since $f(Y_\alpha)\subset \bigcup_{\beta\leq \alpha}Y_\beta$ 
and $P_\alpha$ is outside of $Z_\alpha$, we conclude that $f(P_\alpha)\subset Y_\alpha$.
For each  isolated $\alpha\leq \tau$ the triple $\{Y_\alpha, P_\alpha, f|_{P_\alpha}\}$ meets
the requirement of Proposition \ref{P:lc}. Therefore, we can fix an ($n+3$)-sized
family ${\mathcal F}_\alpha = \{F^\alpha_1,....,F^\alpha_{n+3}\}$, which is a coloring for $f|_{P_\alpha}$
and $cl_Y(f(F))\cap F=\emptyset$ for every $F\in {\mathcal F}_\alpha$.
Put $F_i = \oplus \{F_i^\alpha: \alpha\leq \tau\}$. 
Each $F_i$ is closed as the union of a discrete family of closed sets. Since $f(F_i^\alpha)\subset Y_\alpha$,
we conclude that $cl_Y(f(F_i^\alpha))$ misses $F_i^\beta$ if $\alpha\not=\beta$. By the choice of ${\mathcal F}_\alpha$,
$cl_Y(f(F_i^\alpha))$ misses $F_i^\alpha$. Therefore, $cl_Y(f(F_i))\cap F_i=\emptyset$.
Thus, ${\mathcal F}_P = \{F_i:i\leq n+3\}$ covers
$\oplus_{\alpha} P_\alpha$ by colors
and $cl_Y(f(F))\cap F=\emptyset$ for every $F\in {\mathcal F}_P$. Thus, ${\mathcal F}_Z\cup {\mathcal F}_P$ 
is a desired coloring of $f$.
\end{proof}

\par\bigskip
The existence of a special coloring in Theorem \ref{THEOREM:PARACOMPACT} implies that $\widetilde f:\beta X\to \beta Y$
is fixed-point free as well. Therefore, we have the following equivalence statement.
\par\bigskip\noindent
\begin{thm}\label{THEOREM:EQUIVALENCE}
Let $f \colon X \to Y$ be a continuous map  defined on a closed subspace $X$ of a locally compact and 
paracompact space $Y$ of dimension $\dim Y \leq n$. Then $f$ is  fixed-point free iff $\widetilde f:\beta X\to \beta Y$ is
fixed-point free.
\end{thm} 

\par\bigskip
We finish this section by giving an estimate for the chromatic number for maps as in the hypothesis
of Theorem \ref{THEOREM:PARACOMPACT}.

\par\bigskip\noindent
\begin{thm}\label{THEOREM:GENERAL}
Any continuous fixed-point free map $f \colon X \to Y$, defined on a closed subspace $X$ of a locally compact and 
paracompact space $Y$ of dimension $\dim Y \leq n$ is colorable in at most $(n+3)$ colors.
\end{thm}
\begin{proof}
By Theorem \ref{THEOREM:EQUIVALENCE}, $\widetilde{f}$ has no fixed points. 
Consequently, by Proposition \ref{P:lc}, $\widetilde{f}$ is colorable in at most $n+3$ colors, 
which implies that the chromatic number of $f$ also does not exceed $n+3$.
\end{proof}


\section{Examples and Comments}\label{S:examples}

As mentioned in the Introduction Mazur constructed a continuous fixed-point free 
non-colorable map of the space of irrationals into itself. The space $P$ of irrational 
numbers is the first element in the hierarchy $\{ N_{n}^{2n+1} \colon n \in \omega\}$ of 
$n$-dimensional universal N\"{o}beling spaces (i.e. $P = N_{0}^{1}$). 

\begin{pro}\label{P:nobeling}
Let $n \in \omega$. There exists a continuous fixed-point free non-colorable self-map of the 
$n$-dimensional universal N\"{o}beling space $N_{n}^{2n+1}$.
\end{pro}
\begin{proof}
As noted, for $n = 0$ this is known. Let $n > 0$ and $f \colon N_{0}^{1} \to N_{0}^{1}$ be a 
continuous fixed-point free non-colorable map. Embed $N_{0}^{1}$ into $N_{n}^{2n+1}$ as a closed subspace. 
By Lemma \ref{L:emb+1}, since $\dim N_{n}^{2n+1} =n$ and $N_{n}^{2n+1} \in AE(n)$, we can embed 
$N_{0}^{1}$ into the product $N_{n}^{2n+1} \times [0,\infty)$ as a closed subspace in such a way that 
$f = g|N_{0}^{1}$, where $g \colon N_{n}^{2n+1}\times [0,\infty) \to N_{n}^{2n+1}\times [0,\infty)$ is a 
fixed-point free map. Let $p \colon N_{n}^{2n+1} \to N_{n}^{2n+1}\times [0,\infty)$ be an $n$-soft map. 
Choose a section $i \colon N_{0}^{1} \to N_{n}^{2n+1}$ of $p|p^{-1}(N_{0}^{1})$, i.e. $pi = \operatorname{id}_{N_{0}^{1}}$, 
and let $A = i(N_{0}^{1})$. Consider the map $j \colon A \to A$ defined by $j = i\circ f\circ p|A$ and note that 
$p\circ j = p\circ i\circ f\circ p|A = f\circ p|A = g\circ p|A$. Since $p$ is $n$-soft there exists a map 
$r \colon N_{n}^{2n+1}\to N_{n}^{2n+1}$ such that $p\circ r = g\circ p$ and $r|A = j$. The former implies that $r$ is 
fixed-point free. The latter implies that $r$ is not colorable (otherwise $j$ would have been colorable; but since $j$ is 
equivalent to $f$ this is impossible). 
\end{proof}

\begin{pro}\label{P:Hilbert}
There exists a continuous fixed-point free non-colorable self-map of the separable Hilbert space $\ell_{2}\simeq {\mathbb R}^{\omega}$.
\end{pro}
\begin{proof}
Let ${\mathbb S}^{n} = \{ x \in {\mathbb R}^{n+1} \colon || x|| = \frac{1}{n+1}\} \times \{ 0_{i}\}_{i \geq n+2} \subset {\mathbb R}^{\omega}$. 
The map $f({\mathbf x}) = -{\mathbf x}$ is fixed-point free self-map of ${\mathbb P} = {\mathbb R}^{\omega} \setminus \{ {\mathbf 0}\}$ 
and coincides with the antipodal map on each ${\mathbb S}^{n}$. According to Van Douwen's observation 
$\widetilde{g} \colon \beta {\mathbb S} \to \beta{\mathbb S}$, where ${\mathbb S} = \cup {\mathbb S}^{n}$ and 
$g = f|{\mathbb S}$, fixes a point. Since ${\mathbb S}$ is closed in ${\mathbb P}$, it follows that $\widetilde{f}$ 
fixes a point which means that $f$ is not colorable.  Since $\ell_{2}$ is homeomorphic to ${\mathbb P}$ the needed conclusion follows.
\end{proof}

\begin{pro}\label{P:closurefp}
Let $f \colon X\to X$ be a continuous self-map of a finite-dimensional locally compact paracompact space and 
$\widetilde{f} \colon \beta X \to \beta X$ be its extension. Then 
$\operatorname{Fix}(\widetilde{f}) = \operatorname{cl}_{\beta X}\operatorname{Fix}(f)$. 
\end{pro}
\begin{proof}
Obviously $\operatorname{cl}_{\beta X}\operatorname{Fix}(f) \subset \operatorname{Fix}(\widetilde{f})$. 
Suppose that there is $p \in \operatorname{Fix}(\widetilde{f}) \setminus \operatorname{cl}_{\beta X}\operatorname{Fix}(f)$. 
Choose a functionally closed set $Z \subset X$ such that $p \in \operatorname{cl}_{\beta X}Z$ and $Z \cap \operatorname{Fix}(f) = \emptyset$. 
Let $\varphi \colon X \to [0,\infty)$ be a function such that $\varphi^{-1}(0) = Z$. Consider the map 
$g \colon X \times [0,\infty) \to X \times [0,\infty)$ defined by letting
$g(x,r) = (f(x),r + \varphi(x))$, $(x,r) \in X \times [0,\infty)$, and note that $g$ has no fixed points. 
The product $X \times [0,\infty)$ is still finite-dimensional, locally compact and paracompact. Therefore, 
by Theorem \ref{THEOREM:EQUIVALENCE}, $\widetilde{g} \colon \beta(X\times [0,\infty)) \to \beta (X \times [0,\infty))$ 
has no fixed points. On the other hand, identifying $Z$ with $Z \times \{ 0\}$, we see that $f|Z = g|Z$ and consequently 
$\widetilde{g}(p) = \widetilde{f}(p) = p$. This contradiction completes the proof. 
\end{proof}

The next observation shows that assumption of closedness of $X$ (in $Y$) in Theorem \ref{THEOREM:PARACOMPACT} 
(as well as in Propositions \ref{P:chromatic} and \ref{P:lc}) cannot be dropped.

\begin{pro}\label{P:open}
There exists a continuous non-colorable fixed-point free map $g \colon X \to Y$, defined on an open subspace $X$ of a 
zero-dimensional metrizable compactum $Y$.
\end{pro}
\begin{proof}
Let $f \colon P \to P$ be a continuous non-colorable fixed-point free self-map of 
the space $P$ of irrational numbers. Since the extension of $f$ onto the Stone-\v{C}ech 
compactification $\beta P$ of $P$ has fixed points, there exist a zero-dimensional metrizable 
compactification $Y$ of $P$ and a map $\widetilde{f} \colon Y \to Y$ such that $\widetilde{f}|P = f$ 
and $\operatorname{Fix}(\widetilde{f}) \not= \emptyset$. Set $X = \widetilde{f}^{-1}(\operatorname{Fix}(\widetilde{f}))$ and 
$g = \widetilde{f}|X$.
\end{proof} 

The following statement can be considered as a converse (among Polish spaces) of \cite[Theorem 3.3]{KS}.

\begin{pro}\label{P:iff}
Let $X$ be a zero-dimensional Polish space. Then the following conditions are equivalent:
\begin{enumerate}
  \item Every continuous fixed-point free self-map of $X$ is colorable;
  \item $X$ is sigma-compact.
\end{enumerate}
\end{pro}
\begin{proof}
$(1) \Longrightarrow (2)$. By $LC(X)$ denote the union of all open compact subsets of $X$. 
Let $X^{0} = LC(X)$, $X^{\alpha +1} = X^{\alpha} \cup LC(X\setminus X^{\alpha})$ and 
$X^{\alpha} = \cup \{ X^{\beta} \colon \beta < \alpha\}$ for a limit ordinal. Note that 
$\{ X^{\alpha} \colon \alpha < \omega_{1}\}$ is an increasing union of open subsets of $X$ 
and consequently $X^{\beta} = X^{\alpha}$ for some $\beta < \omega_{1}$ and all $\alpha > \beta$. 
Clearly $Y = X\setminus X^{\beta}$ has no points of local compactness. Note also that $X^{\beta}$ 
is sigma-compact. Assuming that $X$ is not sigma-compact, it follows that $Y \not= \emptyset$. Since $Y$ 
is completely metrizable (as a closed subspace of $X$) and zero-dimensional, we conclude that $Y$ is 
homeomorphic to the space of irrational numbers.
Let $g \colon Y \to Y$ be a continuous non-colorable fixed-point free map (Mazur's example) and $r \colon X \to Y$ be a retraction. 
Then the composition $f = g\circ r \colon X \to X$ is a non-colorable fixed-point free map contradicting (1).

$(2) \Longrightarrow (1)$. This follows from \cite{KS}.
\end{proof}

\par\bigskip
Our main theorem for locally compact Lindel\"of spaces of finite dimension and the Krawczyk-Stepr\={a}ns result for
zero-dimensional $\sigma$-compact spaces \cite[Theorem 3.3]{KS} motivate the following
question.

\par\bigskip\noindent
\begin{que}\label{Q:generalizeKSandRCH}
Let $X$ be a $\sigma$-compact space of finite dimension.
Is every continuous fixed-point free self-map on $X$ colorable?
\end{que}

\end{document}